\documentclass[12pt]{article}
\usepackage{amsmath,amssymb,amsthm,comment,cite}

\newtheorem{theorem}{Theorem}[section]
\newtheorem{lemma}[theorem]{Lemma}
\newtheorem{corollary}[theorem]{Corollary}
\newtheorem{conjecture}[theorem]{Conjecture}
\newtheorem{proposition}[theorem]{Proposition}

\def\barr{\begin{array}}
\def\earr{\end{array}}

\title{On a group-theoretical generalization of the Gauss formula}
\author{Georgiana Fasol\u a and Marius T\u arn\u auceanu}
\date{August 13, 2022}

\begin{document}

\maketitle

\begin{abstract}

In this paper, we discuss a group-theoretical generalization of the well-known Gauss formula involving the function
that counts the number of automorphisms of a finite group. This gives several characterizations of finite cyclic groups.
\end{abstract}
\smallskip

{\small
\noindent
{\bf MSC 2020\,:} Primary 20D60, 11A25; Secondary 20D99, 11A99.

\noindent
{\bf Key words\,:} Gauss formula, Euler's totient function, automorphism group, finite group, cyclic group, abelian group.}

\section{Introduction}

The \textit{Euler's totient function} (or, simply, the \textit{totient function}) $\varphi$ is one of the most famous functions in number theory. The totient $\varphi(n)$ of a positive integer $n$ is defined to be the number of positive integers less than or equal to $n$ that are coprime to $n$. In algebra this function is important mainly because it gives the order of the group of units in the ring $(\mathbb{Z}_n,+,\cdot)$. Also, $\varphi(n)$ can be seen as the number of generators or as the number of automorphisms of the cyclic group $(\mathbb{Z}_n,+)$. Note that there exist a lot of identities involving the totient function. One of them is the \textit{Gauss formula}
\begin{equation}
\sum_{d|n}\varphi(d)=n,\, \forall\, n\in\mathbb{N}^*.
\end{equation}

In the last years there has been a growing interest in extending arithmetical notions to finite groups (see e.g. \cite{1,2,9,10,13,14}). Following this trend, we remark that (1) can be rewritten as

\begin{equation}
\sum_{H\leq\mathbb{Z}_n}|{\rm Aut}(H)|=|\mathbb{Z}_n|,\, \forall\, n\in\mathbb{N}^*.\nonumber
\end{equation}It suggests us to consider the functions

\begin{equation}
S(G)=\sum_{H\leq G}|{\rm Aut}(H)| \mbox{ and } f(G)=\frac{S(G)}{|G|}\nonumber
\end{equation}for any finite group $G$. Thus the classical Gauss formula becomes
\begin{equation}
f(\mathbb{Z}_n)=1,\, \forall\, n\in\mathbb{N}^*.
\end{equation}

The main goal of our paper is to study the above function $f$. We start by observing that it is multiplicative, i.e. if $G_i$, $i=1,2,\dots,m$, are finite groups of coprime orders, then we have
\begin{equation}
f\left(\prod_{i=1}^m G_i\right)=\prod_{i=1}^m f(G_i).\nonumber
\end{equation}This implies that the computation of $f(G)$ for a finite nilpotent group $G$ is reduced to $p$-groups.
\smallskip

Our first theorem shows that the cyclic groups are in fact the unique groups satisfying (2).

\begin{theorem}
Let $G$ be a finite group. Then $f(G)\geq 1$, and we have equality if and only if $G$ is cyclic.
\end{theorem}

The above theorem leads to the following natural question:\vspace{1mm}
\begin{center}
\textit{Is there a minimum of $f$\! on the class of finite non-cyclic groups}?\vspace{1mm}
\end{center}In what follows, we will give some partial answers to this question, by finding the constant $c$ for several
particular classes of finite non-cyclic groups.

\begin{proposition}
For any finite non-cyclic group $G$, we have
\begin{equation}
f(G)\geq 1+\frac{1}{|Z(G)|}\,.\nonumber
\end{equation}In particular, if $G$ is centerless, then $f(G)\geq 2$.
\end{proposition}\newpage

Note that Proposition 1.2 implies
\begin{equation}
f(G)\geq 1+\frac{4}{|G|}\,,\nonumber
\end{equation}for any finite non-abelian group $G$. Also, by the proof of Proposition 2.1, we infer that small values of $f(G)$ are obtained for finite non-cyclic groups $G$ with many cyclic subgroups. This leads to the next proposition. We recall that a finite group $G$ is called a \textit{minimal non-cyclic group} if $G$ is not cyclic, but all proper subgroups of $G$ are cyclic.

\begin{proposition}
Let $G$ be a finite minimal non-cyclic group. Then $f(G)\geq 2$, and we have equality if and only if $G\cong\mathbb{Z}_3\rtimes\mathbb{Z}_{2^n}$, $n\in\mathbb{N}^*$.
\end{proposition}

The following theorem shows that the constant $c$ can be taken $\frac{5}{2}$ for finite abelian groups.

\begin{theorem}
Let $G$ be a finite non-cyclic abelian group. Then $f(G)\geq\frac{5}{2}$\,, and we have equality if and only if $G\cong\left(\mathbb{Z}_2\times\mathbb{Z}_2\right)\times\mathbb{Z}_n$, where $n$ is an odd positive integer.
\end{theorem}

Inspired by the above results, we came up with the following conjecture, which we have verified by computer for many classes of finite groups.

\begin{conjecture}
$2$ is the second smallest value of the function $f$.
\end{conjecture}

For the proof of Theorem 1.4, we will need to know the number of auto\-morphisms of a finite abelian $p$-group. This has been explicitly computed e.g. in \cite{3,7,12}.

\begin{theorem}
Let $G\cong\prod_{i=1}^k \mathbb{Z}_{p^{n_i}}$ be a finite abelian $p$-group, where $1\leq n_1\leq n_2\leq...\leq n_k$. Then
\begin{equation}
|{\rm Aut}(G)|=\prod_{i=1}^k (p^{a_i}-p^{i-1})\prod_{u=1}^k p^{n_u(k-a_u)}\prod_{v=1}^k p^{(n_v-1)(k-b_v+1)}\,,
\end{equation}where
\begin{equation}
a_r=max\{s\mid n_s=n_r\} \mbox{ and } b_r=min\{s\mid n_s=n_r\}\,,\, r=1,2,...,k\,.\nonumber
\end{equation}In particular, for $k=2$, we have
\begin{equation}
|{\rm Aut}(G)|=(p-1)^2\left(p+1\right)^{\left[\frac{n_1}{n_2}\right]}p^{3n_1+n_2-\left[\frac{n_1}{n_2}\right]-2}.
\end{equation}
\end{theorem}\newpage

We end this paper by indicating a list of open problems concerning our previous results.

\bigskip\noindent{\bf Problem 1.} Determine all finite groups $G$ such that $f(G)=2$.

\bigskip\noindent{\bf Problem 2.} Find the minimum of $f$ on the class of finite non-cyclic $p$-groups.

\bigskip\noindent{\bf Problem 3.} Is ${\rm Im}(f)$ dense in the interval $[2,+\infty)$?
\smallskip

Since $f(D_{2n})=\frac{n+1}{2}$\,, for any odd integer $n\geq 3$, it follows that $\mathbb{N}^*\subseteq{\rm Im}(f)$.
Thus the function $f$ takes arbitrarily large values.
\smallskip

Most of our notation is standard and will usually not be repeated here. For basic notions and results on groups we refer the reader to \cite{8}.

\section{Proof of the main results}

First of all, we prove Theorem 1.1.

\bigskip\noindent{\bf Proof of Theorem 1.1.} Let $C(G)$ be the poset of cyclic subgroups of $G$. For every divisor $d$ of $|G|$, we denote by $n_d$ the number of cyclic subgroups of order $d$ of $G$ and by $n'_d$ the number of elements of order $d$ in $G$. Then we have
\begin{equation}
n'_d=n_d\varphi(d)\nonumber
\end{equation}because a cyclic subgroup of order $d$ contains $\varphi(d)$ elements of order $d$. One obtains
\begin{align*}
S(G)&=\sum_{H\leq G}|{\rm Aut}(H)|\geq\sum_{H\in C(G)}|{\rm Aut}(H)|=\sum_{H\in C(G)}\varphi(|H|)\\
&=\sum_{d\mid n}\!\sum_{H\in C(G),\, |H|=d}\!\!\!\varphi(d)=\sum_{d\mid n}n_d\varphi(d)=\sum_{d\mid n}n'_d=|G|\,,\nonumber
\end{align*}which shows that $f(G)\geq 1$. Moreover, we have equality if and only $C(G)=L(G)$, i.e. if and only if $G$ is cyclic, as desired.\qed
\bigskip

\bigskip\noindent{\bf Proof of Proposition 1.2.} By the proof of Theorem 1.1, it follows that
\begin{equation}
f(G)=1+\frac{1}{|G|}\sum_{H\notin C(G)}|{\rm Aut}(H)|\,,\nonumber
\end{equation}for any finite group $G$. If $G$ is non-cyclic, then we get
\begin{equation}
f(G)\geq 1+\frac{|{\rm Aut}(G)|}{|G|}\geq 1+\frac{|{\rm Inn}(G)|}{|G|}=1+\frac{1}{|Z(G)|}\,,\nonumber
\end{equation}as desired.\qed
\bigskip

Note that the ratio $r(G)=\frac{|{\rm Aut}(G)|}{|G|}$ is $\geq 1$ for many classes of finite groups $G$. However, there are examples of finite groups $G$ with $Z(G)\neq 1$, but of arbitrarily small $r(G)$ (see e.g. \cite{4,5,6}).

\bigskip\noindent{\bf Proof of Proposition 1.3.} By a classical result of Miller and Moreno \cite{11}, a finite minimal non-cyclic group is of one of the following types:
\begin{itemize}
\item[1)] $\mathbb{Z}_p\times\mathbb{Z}_p$, where $p$ is a prime;
\item[2)] $Q_8$;
\item[3)] $\langle a,b\mid a^p=b^{q^n}=1,\, b^{-1}ab=a^r\rangle$, where $p,q$ are distinct primes and $r\not\equiv 1\, ({\rm mod}\, p)$, $r^q\equiv 1\, ({\rm mod}\, p)$.
\end{itemize}For these groups we easily obtain:
\begin{itemize}
\item[1)] $f(\mathbb{Z}_p\times\mathbb{Z}_p)=1+\frac{(p+1)(p-1)^2}{p}>2$, for all primes $p$,
\item[2)] $f(Q_8)=4>2$,
\item[3)] $f(\langle a,b\mid a^p=b^{q^n}=1,\, b^{-1}ab=a^r\rangle)=1+\frac{p-1}{q}\geq 2$ because $q\mid p-1$.
\end{itemize}Moreover, we have $f(G)=2$ if and only if $G$ is of type 3) and $p=3$, $q=2$, i.e. if and only if $G\cong\mathbb{Z}_3\rtimes\mathbb{Z}_{2^n}$, $n\in\mathbb{N}^*$. This completes the proof.\qed
\bigskip

Before proving Theorem 1.4, we establish two auxiliary results.

\begin{lemma}
Let $G$ be a non-cyclic abelian $p$-group of order $p^n$. Then $|{\rm Aut}(G)|\geq p^n(p-1)^2$.
\end{lemma}

\begin{proof}
Let $G\cong\prod_{i=1}^k \mathbb{Z}_{p^{n_i}}$, where $k\geq 2$ and $1\leq n_1\leq n_2\leq...\leq n_k$. For $k=2$, we have
\begin{align*}
|{\rm Aut}(G)|&=(p-1)^2\left(p+1\right)^{\left[\frac{n_1}{n_2}\right]}p^{3n_1+n_2-\left[\frac{n_1}{n_2}\right]-2}\\
&\geq (p-1)^2p^{3n_1+n_2-2}\\
&\geq (p-1)^2p^{n_1+n_2}\nonumber
\end{align*}by (4).\newpage

For $k\geq 3$, we will use the general formula (3). Observe that
\begin{equation}
i\leq a_i\leq k \mbox{ and } 1\leq b_i\leq i, \forall\, i=1,2,\dots,k.\nonumber
\end{equation}Then
\begin{align*}
|{\rm Aut}(G)|&=\prod_{i=1}^k (p^{a_i}-p^{i-1})\prod_{u=1}^k p^{n_u(k-a_u)}\prod_{v=1}^k p^{(n_v-1)(k-b_v+1)}\\
&\geq\prod_{i=1}^k (p^i-p^{i-1})\prod_{v=1}^k p^{(n_v-1)(k-v+1)}\\
&=(p-1)^kp^S,\nonumber
\end{align*}where
\begin{equation}
S=\sum_{i=1}^k\left[i-1+(n_i-1)(k-i+1)\right]=(n_1-1)k+\sum_{i=2}^k n_i(k-i+1).\nonumber
\end{equation}Since $(p-1)^k>(p-1)^2$, it suffices to show that
\begin{equation}
S\geq n_1+n_2+\dots+n_k,\nonumber
\end{equation}which is equivalent to
\begin{equation}
\sum_{i=1}^{k-1} n_i(k-i)\geq k.\nonumber
\end{equation}This is true because for $k\geq 3$ we have
\begin{equation}
\sum_{i=1}^{k-1} n_i(k-i)\geq\sum_{i=1}^{k-1}(k-i)=\frac{k(k-1)}{2}\geq k,\nonumber
\end{equation}completing the proof.
\end{proof}

\begin{corollary}
Let $G$ be a finite abelian group. Then $|{\rm Aut}(G)|\geq\varphi(|G|)$, and we have equality if and only if $G$ is cyclic.
\end{corollary}

\begin{proof}
Obviously, the result is reduced to finite abelian $p$-groups. If $G$ is a non-cyclic abelian $p$-group of order $p^n$, then Lemma 2.1 leads to
\begin{equation}
|{\rm Aut}(G)|\geq p^n(p-1)^2>p^{n-1}(p-1)=\varphi(|G|).\nonumber
\end{equation}The proof is completed by the well-known fact that $|{\rm Aut}(G)|=\varphi(|G|)$, for all finite cyclic groups $G$.
\end{proof}

We are now able to prove Theorem 1.4.\newpage

\bigskip\noindent{\bf Proof of Theorem 1.4.} Let $G\cong\prod_{i=1}^m G_i$ be the decomposition of $G$ as a direct product of abelian $p$-groups. Then
\begin{equation}
f(G)=\prod_{i=1}^m f(G_i).\nonumber
\end{equation}Since $G$ is not cyclic, at least one of the groups $G_i$, $i=1,2,\dots,m$, is not cyclic. On the other hand, we already know that $f(G_i)\geq 1$, for all $i$, by Theorem 1.1. These shows that it suffices to prove the inequality $f(G)\geq\frac{5}{2}$ for non-cyclic abelian $p$-groups.

Assume that $G\cong\prod_{i=1}^k \mathbb{Z}_{p^{n_i}}$ is an abelian $p$-group, where $k\geq 2$ and $1\leq n_1\leq n_2\leq...\leq n_k$. By induction on $n=n_1+n_2+\cdots+n_k$, we will prove that
\begin{equation}
f(G)\geq 1+\frac{(p+1)(p-1)^2}{p}\,.
\end{equation}For $n=2$, we have $G\cong\mathbb{Z}_p\times\mathbb{Z}_p$ and so
\begin{equation}
f(G)=1+\frac{|{\rm Aut}(G)|}{|G|}=1+\frac{|{\rm GL}(2,p)|}{p^2}=1+\frac{(p+1)(p-1)^2}{p}\,.\nonumber
\end{equation}Suppose now that $n\geq 3$ and that (5) holds for any non-cyclic abelian $p$-group of order $p^{n-1}$. Since $G$ is not cyclic, it possesses at least $p+1$ maximal subgroups $M_0,M_1,\dots,M_p$. Moreover, at least one of them, say $M_0$, is non-cyclic. Then Lemma 2.1 and Corollary 2.2 imply that
\begin{align*}
S(G)&\geq S(M_0)+|{\rm Aut}(G)|+\sum_{i=1}^p |{\rm Aut}(M_i)|\\
&\geq S(M_0)+p^n(p-1)^2+p\,\varphi(p^{n-1}).\nonumber
\end{align*}By using the inductive hypothesis, we get
\begin{align*}
f(G)&\geq\frac{1}{p}\,f(M_0)+(p-1)^2+\frac{p-1}{p}\\
&\geq\frac{1}{p}\left[1+\frac{(p+1)(p-1)^2}{p}\right]+(p-1)^2+\frac{p-1}{p}\\
&=1+\frac{(p^2+p+1)(p-1)^2}{p^2}\\
&>1+\frac{(p+1)(p-1)^2}{p}\,,\nonumber
\end{align*}as desired.\newpage

We can easily see that the minimum value of the right side of (5) is $\frac{5}{2}$\,, and it is attained for $p=2$. Hence we
have $f(G)\geq\frac{5}{2}$\,, with equality if and only if $p=2$ and $G\cong\mathbb{Z}_2\times\mathbb{Z}_2$. This completes the proof.\qed
\bigskip

\noindent{\bf Acknowledgments.} The authors are grateful to the reviewers for their remarks which improve the previous version of the paper.

\vspace*{3ex}
\small

\begin{minipage}[t]{7cm}
Georgiana Fasol\u a \\
Faculty of  Mathematics \\
"Al.I. Cuza" University \\
Ia\c si, Romania \\
e-mail: \!{\tt georgiana.fasola@student.uaic.ro}
\end{minipage}
\hfill\hspace{20mm}
\begin{minipage}[t]{5cm}
Marius T\u arn\u auceanu \\
Faculty of  Mathematics \\
"Al.I. Cuza" University \\
Ia\c si, Romania \\
e-mail: \!{\tt tarnauc@uaic.ro}
\end{minipage}


\bigskip

\begin{thebibliography}{10}
\bibitem{1} S.J. Baishya, {\it Revisiting the Leinster groups}, C.R. Math. Acad. Sci. Paris {\bf 352} (2014), 1-6.
\bibitem{2} S.J. Baishya and A.K. Das, {\it Harmonic numbers and finite groups}, Rend. Semin. Mat. Univ. Padova {\bf 132} (2014), 33-43.
\bibitem{3} J.N.S. Bidwell, M.J. Curran and D.J. McCaughan, {\it Automorphisms of direct products of finite groups}, Arch. Math. {\bf 86} (2006), 481-489.
\bibitem{4} J.N. Bray and R.A. Wilson, {\it On the orders of automorphism groups of finite groups}, Bull. London Math. Soc. {\bf 37} (2005), 381–385.
\bibitem{5} J.N. Bray and R.A. Wilson, {\it On the orders of automorphism groups of finite groups}, II, J. Group Theory {\bf 9} (2006), 537–547.
\bibitem{6} J. Gonz\'{a}lez-S\'{a}nchez and A. Jaikin-Zaipirain, {\it Finite $p$-groups with small automorphism group}, Forum Math. Sigma {\bf 3} (2015), e7.
\bibitem{7} C.J. Hillar and D.L. Rhea, {\it Automorphisms of finite abelian groups}, Amer. Math. Monthly {\bf 114} (2007), 917-923.
\bibitem{8} I.M. Isaacs, {\it Finite group theory}, Amer. Math. Soc., Providence, R.I., 2008.
\bibitem{9} T. De Medts and M. T\u arn\u auceanu, {\it Finite groups determined by an ine\-quality of the orders of their subgroups}, Bull. Belg. Math. Soc. Simon Stevin {\bf 15} (2008), 699-704.
\bibitem{10} T. De Medts and A. Mar\'{o}ti, {\it Perfect numbers and finite groups}, Rend. Semin. Mat. Univ. Padova {\bf 129} (2013), 17-33.
\bibitem{11} G.A. Miller and H.C. Moreno, {\it Non-abelian groups in which every subgroup is abelian}, Trans. Amer. Math. Soc. {\bf 4} (1903), 398-404.
\bibitem{12} A. Sehgal, S. Sehgal and P.K. Sharma, {\it The number of automorphism of a finite abelian group of rank two}, J. Discrete Math. Sci. Cryptogr. {\bf 19} (2016), 163-171.
\bibitem{13} M. T\u arn\u auceanu, {\it A generalization of the Euler's totient function}, Asian-Eur. J. Math. {\bf 8} (2015), no. 4, article ID 1550087.
\bibitem{14} M. T\u arn\u auceanu, {\it Finite groups determined by an inequality of the orders of their subgroups} II, Comm. Algebra {\bf 45} (2017), 4865-4868.
\end{thebibliography}
\end{document}